\documentclass[a4paper,fleqn]{cas-dc}

\usepackage[numbers]{natbib}
\usepackage{amssymb}
\usepackage{enumitem}
\usepackage{amsmath}
\usepackage{caption}
\usepackage{subcaption}
\usepackage{float}
\usepackage{xcolor}
\usepackage[utf8x]{inputenc}
\usepackage{amsthm}
\usepackage{amssymb}
\usepackage{mathtools}
\usepackage{algorithm}

\def\tsc#1{\csdef{#1}{\textsc{\lowercase{#1}}\xspace}}
\tsc{WGM}
\tsc{QE}

\makeatletter
\newcommand{\mathleft}{\@fleqntrue\@mathmargin0pt}
\newcommand{\mathcenter}{\@fleqnfalse}
\makeatother

\newtheorem{theorem}{Theorem}
\newtheorem{prop}{Proposition}
\newtheorem{definition}{Definition}
\newtheorem{assump}{Assumption}
\newtheorem{lemma}{Lemma}
\newdefinition{rmk}{Remark}
\newdefinition{cor}{Corollary}
\newtheorem{prob}{Problem formulation}

\DeclareMathOperator*{\argmin}{arg\,min}
\begin{document}
\let\WriteBookmarks\relax
\def\floatpagepagefraction{1}
\def\textpagefraction{.001}

\shorttitle{RL for Inverse LQ Dynamic Games}    

\shortauthors{E. Martirosyan, M. Cao}  

\title [mode = title]{Reinforcement Learning for Inverse Linear-quadratic Dynamic Non-cooperative  Games}  



%

\author[]{E. Martirosyan}[style=chinese]






\affiliation[]{organization={Engineering and Technology Institute Groningen, University of Groningen},
            addressline={Nijenborgh 4}, 
            city={Groningen},
            postcode={9712CP}, 
            state={},
            country={Netherlands}}

\author[]{M. Cao}[style=chinese]








\begin{abstract}
In this paper, we address the inverse problem in the case of linear-quadratic discrete-time dynamic non-cooperative games.  Given feedback laws of players that are known to be a Nash equilibrium pair for a discrete-time linear system, we want find cost function parameters for which the observed feedback laws are optimal and stabilizing. Using the given feedback laws, we introduce a model-based algorithm that generates cost function parameters solving the problem. We provide theoretical results that guarantee the convergence and stability of the algorithm as well as the way to generate new games with necessary properties without requiring to run the complete algorithm repeatedly . Then the algorithm is extended to a model-free version that uses data samples generated by unknown dynamics and has the same properties as the model-based version. Simulation results validate the effectiveness of the proposed algorithms.
\end{abstract}




\maketitle

\section{Introduction}

Game theory involves multi-person decision making. It is dynamic if the order in which the decisions are made is important, and it is non-cooperative if each person involved pursues his or her own interests which are partly conflicting with others' \cite{bacsar1998dynamic}. Introduction of differential games in \cite{isaacs_differential_1965} (the word differential refers to dynamic games with continuous-time dynamics), attracted significant attention due to its connection to the optimal control theory. Linear-quadratic (LQ) non-cooperative games were shown to be suitable for modeling human-machine interaction \cite{flad2017cooperative}, \cite{na2014game}. Also, dynamic games found its application in collision avoidance \cite{mylvaganam2017differential} and formation control problems \cite{gu2007differential}. Although most of the works focus on finding an optimal behavior for a given game, recently a significant attention emerged to the problems where given a desired behavior, the goal is to construct a cost function for which that behavior is optimal. For example, the inverse problems in game theoretical setting were studied for human behavior identification during interaction with an automatic controller \cite{rothfuss2017inverse} or identification of biological systems behavior \cite{molloy2018inverse}. 

Inverse optimal control (IOC) is a field with a long history with one of the earliest work published in 1966 \cite{anderson1966inverse}. The focus of IOC is to develop mathematical models and algorithms for inferring the objectives and constraints of a system in view of observed behavior. Another, closely related field is inverse reinforcement learning (IRL) \cite{ng2000algorithms} where the same task is accomplished but in the framework of Markov decision process. There already has been done a significant work on the inverse dynamic games. Some of them use purely IRL approaches, while others are based on IOC. For example, \cite{kopf2017inverse} solves finite-time LQ using an extended version of maximum entropy IRL \cite{ziebart2008maximum}. While in \cite{molloy2017inverse} and \cite{molloy2019inverse}, establishing connection between inverse optimal control and inverse dynamic games, solution to the inverse open-loop differential is provided given the system dynamics and the knowledge of the basis functions that are used to parameterize the cost functions of the players. In \cite{rothfuss2017inverse}, a model-based solution to the problem of identifying the cost function for one of the player. \cite{inga2019solution} solves the problem in the infinite horizon setting assuming the knowledge of the dynamics. More results have been established in the areas of imitation problems/apprentice games  \cite{lian2022data, lian2021inverse, lian2022inverse, lian2023off, xue2023inverse} which are fundamentally similar to the inverse dynamic game problems. Since the dynamic games are closely related to optimal control problems, useful results might be found in the literature dedicated to tracking and inverse optimal control problems which consider a single control input systems (or a single object of optimization, i.e., cost function)  \cite{xue2021inversec}, \cite{xue2021inverse}. Most of the mentioned works deal with continuous time systems. In our work, the model-based algorithm that finds the cost function for all the players in the case of the discrete-time games is established. Then, the algorithm is extended to the model free setting assuming that instead of game's dynamics, some trajectories generated by the unknown dynamics can be observed. We also show how an unlimited amount of games can be generated without the require to reuse the algorithm. 

The paper is structured as follows. Section \ref{PF} shows preliminary results on LQ non-cooperative $N$-player dynamic games and formulates the problem addressed in the paper. In section \ref{MB}, we describe each step of the model-based algorithm and describe its analytical properties as well as characterize possible solutions of the inverse problem. Section \ref{MF} extends the model-based algorithm to a model-free version that allows to solve the problem without using the system dynamics. In section \ref{SIMS}, we provide simulation results that validate the effectiveness of the proposed algorithms. Finally, section \ref{CON} concludes the paper with suggestions on the possible future research.








\section{Problem formulation}\label{PF}
In this section, LQ non-cooperative discrete time dynamic games are introduced and stationary linear feedback Nash equilibrium is defined. We clarify what an optimal behavior for the game is and introduce the inverse problem. To keep it clear, we provide used notations below. 

\textit{Notations}: For a matrix $P\in\mathbb{R}^{m\times n}$, $P^k$, $P^{(k)}$ denote $P$ to the power of $k$, and matrix $P$ at the $k$-th iteration, respectively. In addition, $P>0$ and $P\geq0$, denote positive definiteness and positive semi-definiteness of matrix $P$, respectively. $I_k$ and $\mathbf{0}_k$ is the $k\times k$ identity matrix and $k\times k$ zero matrix, respectively. $\mathcal{N}=\{1,\dots,N\}$ denotes the set of $N$ players. The notation $(u_i,u_{-i})$ and $(K_i,K_{-i})$ denote a control input and feedback law profile, respectively; $u_i$ and $K_i$ are the control input and feedback law of player $i$, respectively, and $u_{-i}$ and $K_{-i}$ are the sets of control inputs and feedback laws of the rest $N-1$ players, respectively.

\subsection{LQ discrete time dynamic games}
Consider a differential game with discrete-time system dynamics with $N$ players
\begin{align}\label{sysdyn}
    x(k+1) &= A x(k) + \sum_{j=1}^N B_i u_i(k),
\end{align}
where $x(k)\in\mathbb{R}^n$ and $u_i(k)\in\mathbb{R}^{m_i}$ is a control input of player $i\in\mathcal{N}$; plant matrix $A$, control input matrices $B_i$ have appropriate dimensions. For the given discrete-time system, we make the following assumption.
\begin{assump}
The system \eqref{sysdyn} is stabilizable, i.e., there exists a control sequence that stabilizes the system asymptotically.
\end{assump}
We consider that the players select their control to be linear time-invariant feedback laws of the form 
\begin{equation}\label{fbl}
u_i(k) = - K_i x(k),\quad i\in\mathcal{N}.
\end{equation}

From this end, to ease the readability of the paper, we do not write "for $i\in\mathcal{N}$" to every mathematical expression that is valid for every player. All the equations, inequalities and update rules written further imply \textit{that} unless something specific is mentioned. 

In the game, we restrict the admissible controllers to belong to the following set 
\begin{align}\label{stabset}
    \begin{split}
        \mathcal{K} = \{(K_1,\dots,K_N) | \sigma(A - \sum_{j=1}^N B_j K_j) | \in \mathbb{C}_d\},
    \end{split}
\end{align}
   where $\sigma(\cdot)$ and $\mathbb{C}_d$ denote the set of eigenvalues and set $\{\lambda\in\mathbb{C}: |\lambda< 1|\}$, respectively. Control inputs $(u_1,\dots,u_N)$ need to stabilize the system to qualify as NE in this game \cite{engwerda_lq_2005}. This restriction is essential because, as shown in \cite{mage76}, without this restriction it is possible to provide an example where a non-stabilizing feedback yields lower cost for one of the player while another player sticks to the stabilizing feedback law. 

The cost function of each player $i$ is
\begin{align}\label{costfun}
\begin{split}
    & J_i(x(k), u_i(k),u_{-i}(k)) = \sum_{t=k}^\infty r_i(x(k),u_i(k),u_{-i}(k))\\
    & = \sum_{t=k}^\infty (x^\top(t) Q_i x(t) + \sum_{j=1}^N u_j^\top(t) R_{ij} u_j(t)),
\end{split}
\end{align}
where $Q_i\in\mathbb{R}^{n\times n}$, $R_{ij}\in\mathbb{R}^{m_j\times m_j}$ are symmetric (often assumed positive semi-definite matrices \cite{bacsar1998dynamic}) and $R_{ii}\in\mathbb{R}^{m_i\times m_i}$ are symmetric positive definite matrices. Within the game, each player aims to minimize its own cost function, i.e., find a controller $u_i^* = \argmin J_i(x(k),u_i(k),u_{-i}^*(k))$ where $u_{-i}^*$ is a profile of minimizer for the rest $N-1$ players. The solution of the game is a stabilizable tuple $(u_i^*,u_{-i}^*)$ (or, due to \eqref{fbl}, $(K_{i}^*,K_{-i}^*)$, equivalently) which is called Nash Equilibrium (NE) and satisfies 
\begin{align}
    J_i(x(k),u_i^*(k),u_{-i}^*(k)) \leq J_i(x(k),u_i(k),u_{-i}^*(k)),
\end{align}
where $(u_i,u_{-i}^*)$ is generated by the set of feedback laws belonging to the set of admissible controllers \eqref{stabset}.

The value function for player $i$ is defined as 
\begin{align}\label{belleq}
    \begin{split}
        V_i(x(k)) &= \sum_{t=k}^\infty r_i(x(t),u_i(t),u_{-i}(t))\\
        & = r_i(x(k),u_i(k),u_{-i}(k)) + V_i(x(k+1)),
    \end{split}
\end{align}
which for linear systems is set as $V_i(x(k)) = x^\top (k) P_i x(k)$ where $P_i > 0$ \cite{Lewis2012}. Then, the Bellman's optimality equation for the game is given by
\begin{align}\label{belman}
    \begin{split}
        & V_i^*(x(k)) = \min_{u_i} \sum_{t=k}^\infty r_i(x(t),u_i(t),u_{-i}(t))\\
        & = \min_{u_i} \big( r_i(x(k),u_i(k),u_{-i}(k)) + V_i^*(x(k+1)) \big).
    \end{split}
\end{align} 

According to \cite{Vrabie2012}, the Hamiltonian corresponding to each player $i$ is defined as follows 
\begin{align}
\begin{split}
    & H_i(x(k), u_i(k),u_{-i}(k),\Delta V_i(k)) =  \\
    & = r_i(x(k),u_i(k),u_{-i}(k)) + \Delta V_i(k),
\end{split}
\end{align}
where 
\begin{align}
    \begin{split}
        & \Delta V_i(k)) = V_i(x(k+1)) - V_i(x(k)),\\
        & r_i(x(k),u_i(k),u_{-i}(k)) = \\
        & x^\top(k) Q_i x(k) + \sum_{j=1}^N u_j^\top (k) R_{ij} u_j(k).
    \end{split}
\end{align}

From the stationarity condition $\frac{\partial H_i}{\partial u_i} = 0$, one can derive the minimizing control $u_i^*$ input for player $i$ given as
\mathleft
\begin{align}
	\begin{split}
    u_i^*(k) = - (R_{ii} + B_i^\top P_i B_i)^{-1} B_i^\top P_i (A x(k) + \sum_{\substack{j\neq i}}^N B_j u_j^*(k)).
	\end{split}
\end{align}
Considering \eqref{fbl}, one concludes that the optimal feedback law of player $i$ is given by
\mathcenter
\begin{align}\label{eqfb}
    \begin{split}
        K_i^* = (R_{ii} + B_i^\top P_i B_i)^{-1} B_i^\top P_i (A -  \sum_{\substack{j\neq i}}^N  B_j K_j^*).
    \end{split}
\end{align}

From \eqref{belleq}, one can derived a system of generalized algebraic Riccati equations (GAREs) for 
\begin{align}\label{GARE}
	\begin{split}
& Q_i + \sum_{j=1}^N K_j^\top R_{ij} K_j - P_i + \\
& (A - \sum_{j=1}^N B_j K_j)^\top P_i (A - \sum_{j=1}^N B_j K_j) = 0.
	\end{split}
\end{align}
The solution of the system of GAREs $(P_1,\dots,P_N)$ is used to evaluate the set of equilibrium feedback laws $(K_1,\dots,K_N)$ \eqref{eqfb}.

The above coincides with presented in \cite{monti2023feedback} and \cite{song21}, and can be summarized in the following way - if $(K_1,\dots,K_N)$ satisfy \eqref{eqfb} with $(P_1,\dots,P_N)$ being a solution in \eqref{GARE} and belong to the set in \eqref{stabset}, then we can conclude that $(K_1,\dots,K_N)$ is NE.

\subsection{The inverse problem}

This section formulates the inverse problem for LQ discrete time dynamic games. 

Consider an LQ discrete time dynamic game (referred to as the
observed game) with the following system dynamics 

\begin{align}\label{obssysdyn}
    \begin{split}
         x_o(k+1) &= A x_o(k) + \sum_{j=1}^N B_j u_{j,o}(k),
    \end{split}
\end{align}
where $x_o\in\mathbb{R}^n$ is the observed state, $u_{i,o}\in\mathbb{R}^{m_i}$ generated by $K_{i,o}$ is a control input of the player in $\mathcal{N}$. The pair tuple $(u_{1,o},\dots,u_{N,o})$ constitutes NE of a game with the cost function for each player $i\in\mathcal{N}$ is given by
\begin{align}
    \begin{split}
        & J_{i,o} (x(k),u_i(k),u_{-i}(k)) = \\
    & = \sum_{t=k}^\infty (x^\top(t) Q_{i,o} x(t) + \sum_{j=1}^N u_j^\top(t) R_{ij,o} u_j(t))
    \end{split}
\end{align}
with unknown symmetric matrices $Q_{i,o}$ and $R_{ij,o}$ are positive semi-definite; $R_{ii,o}$ is positive definite. Considering that $x_o,u_{1,o},u_{2,o}$ are NE trajectories, one has 
\begin{align}\label{targetfd}
    \begin{split}
    &u_{i,o}(k) = - K_{i,o} x_o(k) = \\ 
    &- (R_{ii} + B_i^\top P_{i,o} B_i)^{-1} B_i^\top P_{i,o} (A x(k) + \sum_{\substack{j\neq i}}^N B_j u_{j,o}(k)), 
    \end{split}
\end{align}
and $P_{i,o}$ being a solution of 
\begin{align}
    \begin{split}
        & Q_{i,o} + \sum_{j=1}^N K_{j,o}^\top R_{ij} K_{j,o} - P_{i,o} + \\
        & (A - \sum_{j=1}^N B_j K_{j,o})^\top P_{i,o} (A - \sum_{j=1}^N B_j K_{j,o}) = 0.
    \end{split}
\end{align}

\begin{definition}
    Sets of matrices $(B_1,\dots,B_N)$, $(Q_1,\dots,Q_N)$ and $(R_{11},\dots,R_{1N},\dots,R_{N1},\dots,R_{NN})$ are denoted by $\mathbf{B}, \mathbf{Q}$ and $\mathbf{R}$, respectively. A tuple $(A,\mathbf{B},\mathbf{Q},\mathbf{R})$ refers to a game with dynamics \eqref{sysdyn} and cost functions parameters \eqref{costfun}. We call two games with the same system dynamics matrices but different cost function parameters $(A,\mathbf{B},\mathbf{Q},\mathbf{R})$ and $(A,\mathbf{B},\mathbf{Q}^\prime,\mathbf{R}^\prime)$ equivalent if both games share a NE, i.e., there exists a tuple $(K_1,,\dots,K_N)$ that is a NE for both games. 
\end{definition}
\begin{prob}
    Given the matrices $A$ and $\mathbf{B}$ that constitute the system dynamics 
and a pair of NE feedback laws $(K_{1,o},\dots,K_{N,o})$, we want to derive a
game equivalent to the $(A,\mathbf{B},\mathbf{Q}_o,\mathbf{R}_o)$ game via a \textbf{model-based} algorithm.
\end{prob}

\begin{prob}
    Given a pair of NE feedback laws $(K_{1,o},\dots,K_{N,o})$, we want to derive a
game equivalent to the $(A,\mathbf{B},\mathbf{Q}_o,\mathbf{R}_o)$ game via a \textbf{model-free} algorithm.
\end{prob}

\begin{rmk}\label{rem1}
   It is known that the target feedback laws can be optimal for different sets of the cost parameters \cite{lancaster1995algebraic}. In the following sections, we will show that there can be infinitely many games that are equivalent to the observed one. 
\end{rmk}
\begin{rmk}
	In fact, instead of pair of equilibrium laws, one can solve the problem given $(x_o,u_{1,o},\dots,u_{N,o})$ NE trajectories 
 \cite{inga2019solution}. Then, a pair of NE feedback laws $(K_{1,o},\dots,K_{N,o})$ can be computed using the trajectories via an estimation procedure, e.g., the least-square method \cite{devore2008probability}. 
\end{rmk}

\section{Problem 1: Model-based algorithm}\label{MB}

In this section, a model-based algorithm is provided. For this algorithm, we assume that matrices that constitute the dynamics of the system are known, i.e., $A,B_1,\dots,B_N$ are known. The model-based algorithm is also a template for the model-free algorithm. Thus, its analytical properties are also valid for the model-free version as it is shown in the further sections. 
The algorithm can be briefly described as follows - firstly, we initialize the cost function parameters  with particular properties. Then algorithm at each iteration solves modified GAREs to compute $P_i$'s that drive $K_i$'s to the target control laws $K_{i,o}$. Then, cost function parameters are updated until a desired tolerance is reached.

\subsection{Model-based algorithm}

For some given cost function parameters, we want to compute $P_i$ such that $K_i$ computed as in \eqref{eqfb} are equal to $K_{i,o}$ for all player $i\in\mathcal{N}$. For that, we need to present the following theoretical result based on \cite{lian2022data}, \cite{lian2023off} and \cite{xue2021inverse} that is necessary for the algorithm establishment. Further, the following notation is used 
\begin{equation}
    A_i = A - \sum_{j\neq i}^N B_j K_{j,o}\,.
\end{equation} 
\begin{lemma}\label{condlem}
If $P_i > 0$ satisfies 
	\begin{align}\label{modGARE}
	\begin{split}
	& Q_i + \sum_{j=1}^N K_{j,o}^\top R_{ij} K_{j,o} - P_i + \\
 & (A_i - B_i K_{i,o})^\top P_i (A_i - B_i K_{i,o}) = 0,
	\end{split}
	\end{align}
and
	\begin{align}\label{modaGARE1}
	\begin{split}
	& Q_i + \sum_{j\neq i} K_{j,o}^\top R_{ij} K_{j,o} + K_i^\top R_{ii} K_i - P_i + \\
 & (A_i - B_i K_i)^\top P_i (A_i - B_i K_i) = 0
	\end{split}
	\end{align}
with $R_{ii} > 0$ for $i\in\mathcal{N}$, then 
\begin{align}\label{eqfb1}
    \begin{split}
        K_i = (R_{ii} + B_i^\top P_i B_i)^{-1} B_i^\top P_i A_i = K_{i,o}\,.
    \end{split}
\end{align}
\end{lemma}
\begin{proof}
Subtracting \eqref{modGARE} from \eqref{modaGARE1}, one has 
\begin{align}
	\begin{split}
  		& 0 = K_i^\top R_{ii} K_i - K_{i,o}^\top R_{ii} K_{i,o} + \\ 
  		& (A_i - B_i K_i)^\top P_i(A_i - B_i K_i) - \\
            &(A_i - B_i K_{i,o})^\top P_i (A_i - B_i K_{i,o}) = 0.
	\end{split}
\end{align}
which, after opening the brackets, gives 
	\begin{align}\label{pr1GARE}
		\begin{split}
		& K_i^\top R_{ii} K_i - K_{i,o}^\top R_{ii} K_{i,o} + K_i^\top B_i^\top P_i B_i K_i - \\
        & A_i^\top P_i B_i K_i - K_i^\top B_i^\top P_i A_i  - K_{i,o}^\top B_i^\top P_i B_i K_{i,o} + \\
        & A_i^\top P_i B_i K_{i,o} + K_{i,o}^\top B_i^\top P_i A_i = 0.
		\end{split}
	\end{align}
Next, from \eqref{eqfb1} multiplying both sides by positive definite matrix $(R_{ii} + B_i^\top P_i B_i)$, one gets
\begin{align}
	\begin{split}
		(R_{ii} + B_i^\top P_i B_i) K_i  = B_i^\top P_i A_i\,.
	\end{split}
\end{align}
Substituting the above into \eqref{pr1GARE} and grouping the terms, one gets
\begin{align}
 	\begin{split}
 		& - K_{i,o}^\top (R_{ii} + B_i^\top P_i B_i) K_{i,o} - K_i^\top (R_{ii} + B_i^\top P_i B_i) K_i + \\ 
		& K_i^\top (R_{ii} + B_i^\top P_i B_i) K_{i,o} + K_{i,o}^\top (R_{ii} + B_i^\top P_i B_i) K_i = 0,
 	\end{split}
\end{align}
which further can be rewritten as 
\begin{align}
	(K_i - K_{i,o}) (R_{ii} + B_i^\top P_i B_i) (K_i - K_{i,o}) = 0.
\end{align}
Since $R_{ii} + B_i^\top P_i B_i$ is positive definite, the above equality holds only if $K_i = K_{i,o}$ which completes the proof.
\end{proof}

With the above Lemma, the algorithm can be established. We solve 
\eqref{modGARE} to drive $P_i$ and, as a result $K_i$, in the direction $K_{i,o}$, and \eqref{modaGARE1} is used to construct the inverse optimal control update of $Q_i$. It starts with initializing the cost functions parameters. We initialize symmetric $Q_i^{(0)} \geq 0$ and $R_{ij} \geq 0$ with $R_{ii} > 0$. Note that the parameters $R_{ij}$ remain unchanged during the iterative procedure because one can scale only $Q_i$ parameters to construct the cost function with desired properties. 

Setting the iteration counter $s=0,1,2,\dots$, one solves the modified GAREs 
	\begin{align}\label{upd1}
	\begin{split}
	&Q_i^{(s)} + \sum_{j=1}^N K_{j,o}^\top R_{ij} K_{j,o} - P_i^{(s+1)} + \\
    &(A_i - B_i K_{i,o})^\top P_i^{(s+1)} (A_i - B_i K_{i,o}) = 0.
	\end{split}
	\end{align}
The above equations for $i\in\mathcal{N}$ constitute a set of discrete Lyapunov (Stein) equations. Each of the equations is guaranteed to have a positive definite solution if
\begin{equation}
   Q_i^{(s)} + \sum_{j=1}^N K_j^{(s)\top} R_{ij} K_j^{(s)} > 0
\end{equation}
and $A_i - B_i K_{i,o}$ is stable. Thus, at each iteration we obtain a unique tuple of positive definite matrices $(P_1^{(s+1)},\dots,P_N^{(s+1)})$.

Then, following the condition in Lemma \ref{condlem}, $Q_i^{(s+1)}$ is updated in the direction of the optimal value for player $i$. Using \eqref{eqfb1}, \eqref{modaGARE1} is rewritten as 
\begin{align}\label{rewrittenmodGare1}
	\begin{split}
		& Q_i + \sum_{j\neq i} K_{j,o}^\top R_{ij} K_{j,o} - P_i + A_i^\top P_i A_i - \\
        &A_i^\top P_i B_i (R_{ii} + B_i^\top P_i B_i)^{-1} B_i^\top P_i A_i = 0.
	\end{split}	
\end{align} 

Setting the iteration counter, one can update $Q_i^{(s+1)}$ using IOC update \cite{had08} as the convex combination between current $Q_i^{(s)}$ and $\Tilde{Q}_i^{(s)}$ associated with \eqref{rewrittenmodGare1} as 
\begin{align}\label{upd21}
	\begin{split}
		& Q_i^{(s+1)} = (1-\alpha_i) Q_i^{(s)} + \alpha_i \Tilde{Q}_i^{(s)} = Q_i^{(s)} + \alpha_i (\Tilde{Q}_i^{(s)} - Q_i^{(s)})\\
        & \Tilde{Q}_i^{(s)} = - \sum_{j\neq i} K_{j,o}^\top R_{ij} K_{j,o} + P_i^{(s+1)} - A_i^\top P_i^{(s+1)} A_i + \\
        &A_i^\top P_i^{(s+1)} B_i (R_{ii} + B_i^\top P_i^{(s+1)} B_i)^{-1} B_i^\top P_i^{(s+1)} A_i\,,
	\end{split}	
\end{align}
where $\alpha_i \in (0,1]$ is a step size (further, section \ref{SIMS} provides additional information on possible step size values). Next, one can deduct the following 
\begin{align}
\begin{split}
	&\Tilde{Q}_i^{(s)} - Q_i^{(s)} =  K_{i,o}^\top R_{ii} K_{i,o} -  A_i^\top P_i^{(s+1)} A_i + \\
    & A_i^\top P_i^{(s+1)} B_i (R_{ii} + B_i^\top P_i^{(s+1)} B_i)^{-1} B_i^\top P_i^{(s+1)} A_i + \\
    & (A_i - B_i K_{i,o})^\top P_i^{(s+1)} (A_i - B_i K_{i,o}).
\end{split}
\end{align}
Opening the brackets, canceling one terms while grouping the other ones, one can rewrite the above to the following form
\begin{align}
    \begin{split}
         & \Tilde{Q}_i^{(s)} - Q_i^{(s)} = K_{i,o}^\top (R_{ii} + B_i^\top P_i^{(s+1)} B_i) K_{i,o} + \\
         & A_i^\top P_i^{(s+1)} B_i  (R_{ii} + B_i^\top P_i^{(s+1)} B_i) B_i^\top P_i^{(s+1)} A_i - \\
         & K_{i,o}^\top B_i^\top P_i^{(s+1)} A_i - A_i^\top P_i^{(s+1)} B_i K_{i,o}\,,
    \end{split}
\end{align}
which can be further simplified as
\begin{equation}
     \Tilde{Q}_i^{(s)} - Q_i^{(s)} =\delta_i^{(s+1)\top} (R_{ii} + B_i^\top P_i^{(s+1)} B_i) \delta_i^{(s+1)},
\end{equation}
where
\begin{equation}\label{deltasm}
    \delta_i^{(s+1)} =  (R_{ii} + B_i^\top P_i^{(s+1)} B_i)^{-1} B_i^\top P_i^{(s+1)} A_i - K_{i,o}
\end{equation}
can be interpreted as a difference measure between $K_{i,o}$ and $K_i$ as if it is calculated with $P_i = P_i^{(s+1)}$ as in \eqref{eqfb1}.

Then, \eqref{upd21} becomes 
\begin{equation}\label{upd2}
    Q_i^{(s+1)} = Q_i^{(s)} + \alpha_i \delta_i^{(s+1)\top} (R_{ii} + B_i^\top P_i^{(s+1)} B_i) \delta_i^{(s+1)}.
\end{equation}
Solving iteratively \eqref{upd1} and \eqref{upd2} enforces the result of Lemma \ref{condlem}. After the procedure is repeated a number of times that gives a desired tolerance, i.e., $\|Q_i^{(s+1)} - Q_i^{(s)}\| \leq \rho_i$ for some small constant $\rho_i > 0$,  the feedback law is computed using $P_i^{(s+1)}$ computed from \eqref{upd1} as
\begin{align}\label{upd3}
	\begin{split}
	K_i = (R_{ii} + B_i^\top P_i^{(s+1)} B_i)^{-1} B_i^\top P_i^{(s+1)} A_i.
	\end{split}
\end{align}

\begin{rmk}
	Note that $R_{ij}$ for all $i,j\in\mathcal{N}$ remain unchanged during the iterative procedure but the scaling of $Q_i^{(s+1)}$ happens in order to achieve $K_i = K_{i,o}$. 
\end{rmk} 
\begin{rmk} 
Although we prove convergence when the number of iterations goes to infinity, we would like to highlight the fact that the result after sufficient finite number of iterations of $K_i^{(s)}$ for $i\in\mathcal{N}$ is slightly different from the one in \eqref{eqfb}. The reason is that each $K_i$ in \eqref{eqfb} includes $K_{-i}$. We calculate $K_i$ in \eqref{upd3} using $K_{-i} = K_{-i,o}$. Essentially, after sufficient numbers of iterations, the difference is negligible.
\end{rmk}

The procedure of the model-based algorithm is summarised below in \textbf{Algorithm~ \ref{mbalg}}.

\begin{algorithm}
\caption{Model-based Algorithm}
\label{mbalg}
    \begin{enumerate}
	\item Initialize $R_{ii} > 0$, $R_{ij} \geq 0$, $Q_i^{(0)} \geq 0$ and the step sizes $\alpha_i\in(0,1]$ for $i,j\in\mathcal{N}$. Set the iteration counter $s = 0$ and the desired tolerance $\rho_i > 0$ for $i\in\mathcal{N}$. \\
	\textbf{For each player $i\in\mathcal{N}$ perform the following}:
	\item Compute $P_i^{(s+1)}$ from \eqref{upd1} and update $Q_i^{(s+1)}$ as in \eqref{upd2}.  
	\item If $\|Q_i^{(s+1)} - Q_i^{(s)}\| \leq \rho_i$ then stop and compute $K_i$ from \eqref{upd3}. Otherwise, set $s+1$ and repeat step 2. 
	\end{enumerate}
\end{algorithm}

\subsection{Convergence and Nash optimality}
In this section, we provide theoretical results on the convergence of the algorithm and Nash optimality of its output in the context of the inverse problem. 

\begin{theorem}\label{thconv}
Given the initialized parameters $R_{ii} > 0$, $R_{ij} \geq 0$ and $Q_i^{(0)} \geq 0$ for $i\in\mathcal{N}$, there exist $\alpha_i > 0$ such that the parameters converge as 
\begin{equation}
 \lim_{s\to\infty} P_i^{(s+1)} = P_i^{\infty},\quad \lim_{s\to\infty} Q_i^{(s+1)} = Q_i^{\infty},   
\end{equation}
and, as a result, the feedback laws $K_i$ satisfy
\begin{equation}
    K_i = (R_{ii} + B_i^\top P_i^{\infty} B_i)^{-1} B_i^\top P_i^{\infty} A_i = K_{i,o}\,.
\end{equation}
\end{theorem}
\begin{proof}
Firstly, the following notation are introduced
\begin{align}
    \begin{split}\label{ktilde}
        & \Delta_i^{(s+1)} =  \delta_i^\top (R_{ii} + B_i^\top P_i^{(s+1)} B_i) \delta_i\,, \\
        & \Tilde{K}_i^{(s+1)} = (R_{ii} + B_i^\top P_i^{(s+1)} B_i)^{-1} B_i^\top P_i^{(s+1)} A_i\,,
    \end{split}
\end{align}
where $\delta_i^{(s+1)}$ in \eqref{deltasm} then can be rewritten as $\delta_i^{(s+1)} =  \Tilde{K}_i^{(s+1)}- K_{i,o}$. Next, consider \eqref{upd2} 
\begin{equation}\label{qupdatee}
	Q_i^{(s+1)} = Q_i^{(s)} + \alpha_i \Delta_i^{(s+1)}. 
\end{equation}
Since computed from \eqref{upd1} $P_i^{(s+1)} > 0$ for $s=0,1,2,\dots$, $R_{ii} + B_i^\top P_i^{(s+1)} B_i > 0$ holds. Thus, one concludes that $\Delta_i^{(s+1)} \geq 0$ for $s=0,1,2,\dots$ with the equality being valid only when 
\begin{equation}
    K_{i,o} = \Tilde{K}_i^{(s+1)} = (R_{ii} + B_i^\top P_i^{(s+1)} B_i)^{-1} B_i^\top P_i^{(s+1)} A_i\,.
\end{equation}
Thus, $Q_i^{(s+1)} > Q_i^{(s)}$ holds while $P_i^{(s+1)}$ is such that $\tilde{K}_i^{(s+1)} \neq K_{i,o}$.  As it is mentioned in Remark \ref{rem1}, for a given fixed $R_{ij}$ parameters, there exists $Q_i$ parameter such that $K_{i,o}$ is optimal feedback law for each player $i\in\mathcal{N}$. Hence, staring with $Q_i^{(0)} \geq 0$, one wants to approach such $Q_i$ from "below" by reaching it with $ \alpha_i$ step size. Thus, $Q_i^{(0)} \leq Q_i$ should be chosen (if $Q_i^{(0)}$ is equal to target $Q_i$, $P_i^{(1)}$ from \eqref{upd1} gives the target feedback law). With selected $\alpha_i$, from \eqref{upd2} one gets new $Q_i^{(1)} > Q_i^{(0)}$. Then, the next $P_i^{(2)}$, computed from \eqref{upd1} with $Q_i^{(1)}$, is such that $\tilde{K}_i^{(2)}$ from \eqref{ktilde} gets closer to $K_{i,o}$. Hence, $\Delta_i^{(2)} < \Delta_i^{(1)}$. By induction, we can conclude that $\Delta_i^{(s+1)} \leq \Delta_i^{(s)}$ and, as a result, $Q_i^{(s+1)} - Q_i^{(s)} < Q_i^{(s)} - Q_i^{(s-1)}$. Thus, $Q_i^{(s+1)}$ increases relative to $Q_i^{(s)}$, but $\Delta_i^{(s+1)}$ decreases relative to $\Delta_i^{(s)}$. It implies from \eqref{ktilde} that $P_i^{(s+1)}$ approaches such a value that $\Tilde{K}_i^{(s+1)}$ in \eqref{ktilde} approaches to $K_{i,o}$. Hence, we can conclude that $\lim_{s\to\infty} Q_i^{(s+1)} = Q_i^\infty$ and, consequently, $\lim_{s\to\infty} P_i^{(s+1)} = P_i^\infty$ where $P_i^\infty$ is such that
\begin{equation}
    (R_{ii} + B_i^\top P_i^{\infty} B_i)^{-1} B_i^\top P_i^{\infty} A_i = K_{i,o}\,.
\end{equation}
From \eqref{upd3}, one concludes $K_i = K_{i,o}$. This proves the convergence of the algorithm and completes the proof. 
\end{proof}

\begin{rmk}
As it is mentioned in the above proof, $Q_i^{(0)}$ needs to be chose such that $Q_i^{(0)} \leq Q_i$. Thus, we suggest to choose $Q_i^{(0)} = \mathbf{0}_n$ or $Q_i^{(0)} = \epsilon I_n$ where epsilon is a small positive constant.
\end{rmk}

\begin{theorem}
Algorithm \ref{mbalg} generates parameters $Q_i^{\infty}$ and $R_{ij}$ for $i,j\in\mathcal{N}$ such that $(A,\mathbf{B},\mathbf{Q}^\infty,\mathbf{R})$ has NE given by $(K_{1,o},\dots,K_{N,o})$. 
\end{theorem}

\begin{proof}
    Theorem \ref{thconv} shows that for each $i\in\mathcal{N}$ $P_i^\infty$ satisfies GARE
\begin{align}\label{result}
    \begin{split}
        & Q_i^\infty + \sum_{j\neq i} K_{j,o}^\top R_{ij} K_{j,o} + K_i^\top R_{ii} K_i - P_i^\infty + \\
        & (A_i - B_i K_i)^\top P_i^\infty (A_i - B_i K_i) = 0,
    \end{split}
\end{align}
where $K_i$ is received from \eqref{upd3}  
\begin{equation}
    K_i = (R_{ii} + B_i^\top P_i^{\infty} B_i)^{-1} B_i^\top P_i^{\infty} A_i = K_{i,o}\,.
\end{equation}
Thus, $P_i^\infty$ is a solution of GARE derived from the Bellman's optimality equation and $K_i$ is such a feedback law that satisfies \eqref{eqfb} derived from the stationarity condition. Thus, we conclude that $(K_{1,o},\dots,K_{N,o})$ is NE feedback laws. This completes the proof.   
\end{proof}

\subsection{Stability and solution characterization}
In this section, we are showing an additional property of the algorithm, namely the stability of the dynamics with particular set of the feedback laws. Also, we provide the characterization of the possible solutions of the inverse dynamic game problem and show how to generate a new equivalent game without requiring to run the algorithm again. 

Firstly, the stability property is discussed. The given tuple $(K_{1,o},\dots,K_{N,o})$ is known to stabilize the system \eqref{obssysdyn} because it is a NE pair. The stability of $(K_1^{(s)},\dots, K_N^{(s)})$ as if we stop at any iteration $s=1,2,\dots$ can not be concluded, i.e., we cannot guarantee that the set of feedback laws at any iteration belongs to the set in \eqref{stabset}. However, when each $K_i^{(s)}$ computed as in \eqref{upd3} approaches the target feedback laws "close enough", let say, at iteration $\bar{s}$, each subsequent iteration is essentially gives stable dynamics $A - \sum_{j=1}^N B_j K_j^{(\bar{s} + s)}$ for $s=0,1,2,\dots$. Another analytical result on stability is provided below. 

\begin{theorem}\label{stabth}
At each iteration $s=0,1,2,\dots$, $P_i^{(s+1)}$ is such that $A_i - B_i K_i$, where $K_i$ computed via \eqref{upd3} with $P_i^{(s+1)}$, stabilize the system. 
\end{theorem}
\begin{proof}
Technically, we want to proof that $\Tilde{K}_i^{(s+1)}$ defined in \eqref{ktilde} is such that
\begin{equation}\label{itdyn}
x(k+1) = (A_i - B_i \Tilde{K}_i^{(s+1)}) x(k) 
\end{equation}
is stable which is equivalent as proving 
\begin{equation}
V_i^{(s+1)}(x(k+1)) - V_i^{(s+1)}(x(k)) < 0,
\end{equation}
where $V_i^{(s+1)}(x(k)) = x^\top (k) P_i^{(s+1)} x(k)$. Hence, one has 
\begin{align}
    \begin{split}
        & V_i^{(s+1)}(x(k+1)) - V_i^{(s+1)} (x(k)) = x^\top (k) \big(- P_i^{(s+1)} + \\
        & (A_i - B_i \Tilde{K}_i^{(s+1)})^\top P_i^{(s+1)} (A_i - B_i \Tilde{K}_i^{(s+1)}\big) x(k).
    \end{split}
\end{align}
Considering \eqref{upd21}, the above can be rewritten as 
\begin{align}\label{stabdif}
    \begin{split}
        & V_i^{(s+1)}(x(k+1)) - V_i^{(s+1)} (x(k)) = x^\top (k) \big(- \Tilde{Q}_i^{(s)} - \\
        & \sum_{j\neq i}^N K_{j,o} R_{ij} K_{j,o} - \Tilde{K}_{i}^{(s+1)\top} R_{ii}\Tilde{K}_{i}^{(s+1)}\big) x(k).
    \end{split}
\end{align}
Next, using \eqref{upd21} and \eqref{qupdatee}, $\Tilde{Q}_i^{(s+1)}$ can be rewritten as 
\begin{align}
    \begin{split}
    &\Tilde{Q}_i^{(s)} = \alpha_i^{-1} ( Q_i^{(s+1)} - (1-\alpha_i) Q_i^{(s)}) = \\
    & \alpha_i^{-1} ( \alpha_i \Delta_i^{(s+1)} + \alpha_i Q_i^{(s)}) =  \Delta_i^{(s+1)} + Q_i^{(s)} \geq 0.
    \end{split}
\end{align} 
Thus, in \eqref{stabdif} one has $\Tilde{Q}_i^{(s)} \geq 0$, $\sum_{j\neq i}^N K_{j,o} R_{ij} K_{j,o}\geq 0$ and $\Tilde{K}_{i}^{(s+1)\top} R_{ii}\Tilde{K}_{i}^{(s+1)} > 0$. Hence, 
\begin{equation}
    V_i^{(s+1)}(x(k+1)) - V_i^{(s+1)} (x(k)) < 0,
\end{equation}
which completes the proof. 
\end{proof}

Next, we provide the result that allows to generate new games without reusing the algorithm. Firstly, note that there exists infinitely many possible combinations of the parameters $(\mathbf{Q}^\prime, \mathbf{R}^\prime)$ that form an equivalent games. All these combination are captured in the following equality 
\begin{align}
	\begin{split}
& (Q_i^\infty - Q_i^\prime) + \sum_{j=1}^N K_j^\top (R_{ij} -  R_{ij}^\prime) K_j - (P_i - P_i^\prime) + \\
& (A - \sum_{j=1}^N B_j K_j)^\top (P_i - P_i^\prime) (A - \sum_{j=1}^N B_j K_j) = 0.
	\end{split}
\end{align}
where $P_i^\prime$ is a solution GARE (with parameters $Q_i^\prime$ and $R_{ij}^\prime$) that needs to satisfy 
\begin{equation}
    K_{i,o} = (R_{ii}^\prime + B_i^\top P_i^\prime B_i)^{-1} B_i^\top P_i^\prime A_i\,.
\end{equation}
Then, the following proposition can be suggested. 
\begin{prop}\label{characprop}
Let Algorithm \ref{mbalg} generated the set of parameters $(\mathbf{Q},\mathbf{R})$ and consider cost function parameters $(\mathbf{Q}^\prime,\mathbf{R}^\prime)$ satisfying 
\begin{equation}\label{change}
Q_i^\prime + \sum_{j\neq i} K_{j,o}^\top R_{ij}^\prime K_{j,o} = Q_i + \sum_{j\neq i} K_{j,o}^\top R_{ij} K_{j,o}.
\end{equation}
Then games $(A,\mathbf{B},\mathbf{Q},\mathbf{R})$ and $(A,\mathbf{B},\mathbf{Q}^\prime,\mathbf{R}^\prime)$ are equivalent.
\end{prop}
\begin{proof}
Changing $Q_i\to Q_i^\prime$ and $R_{ij}\to R_{ij}^\prime$ for $j\neq i$ according to \eqref{change} does not affect $P_i$ in \eqref{GARE} and $K_i$ in \eqref{eqfb}. Thus, $K_i$ remains the same for each $i\in\mathcal{N}$. This completes the proof. 
\end{proof}
The main importance of the proposition is that, after the equivalent game is derived via Algorithm \ref{mbalg}, one can relax assumptions on $R_{ij}$ for $j\neq i$ imposed in Step 1 of the algorithm, i.e., $R_{ij} \geq 0$. 

\section{Problem 2: Model-free algorithm}\label{MF}
The model-free algorithm is based on reinforcement Q-learning presented in \cite{al07} and \cite{kimurasi14} for designing linear discrete-time zero-sum games with application to $H_\infty$-control and optimal tracking control of linear discrete-time systems, respectively. For implementing this, we need to make the following assumption that allows us to collect sufficient data for developing an algorithm that does not use the system dynamics matrices $(A,B_1,\dots,B_N)$. 
\begin{assump}
The system \eqref{obssysdyn} is accessible, i.e., control inputs can be applied to the system for the data collection.
\end{assump}

\subsection{Model-free algorithm: Q-learning}\label{sectionq}
As it is noted in the previous section, the model-based algorithm is the template for the model-free version. The initialization step of the model-free algorithm is the same, i.e., $Q_i^{(0)} \geq 0$, $R_{ij} \geq 0$ and $R_{ii} >0$ for $i,j\in\mathcal{N}$. 

Firstly, referring to Bellman equation \eqref{belman} and considering that $u_i (k) = - K_i x(k)$, we introduce Q-function in the following form 
\begin{align}\label{qfunc}
\begin{split}
& \mathcal{Q}_i (x,u_i,P_i) = x^\top (k) P_i x(k) = \\
& x^\top(k) (Q_i + \sum_{j\neq i} K_j^\top R_{ij} K_j) x(k) + u_i^\top (k) R_{ii} u_i(k) + \\
& x^\top(k) (A_i - B_i K_i)^\top P_i (A_i - B_i K_i) x(k).
\end{split}
\end{align}
Considering \eqref{upd1} multiplied by $x^\top(k)$ and $x(k)$ as in \eqref{qfunc}, the introduced $Q$-function associated with player $i$ can be written as
\begin{align}\label{Hfunc}
	\begin{split}
		& \mathcal{Q}_i(x,u_i,P_i^{(s+1)}) = x^\top(k) P_i^{(s+1)} x(k) = \\
		&	\begin{pmatrix}
		x(k) \\ u_i(k)
		\end{pmatrix}^\top
		\begin{pmatrix}
		H_{ixx}^{(s+1)} & H_{ixu}^{(s+1)}\\
		H_{iux}^{(s+1)} & H_{iuu}^{(s+1)}
		\end{pmatrix}
		\begin{pmatrix}
		x_o(k) \\ u_i(k)
		\end{pmatrix} = \\
         & \bar{x}_i(k) H_i^{(s+1)}\bar{x}_i(k),
	\end{split}
\end{align}
where $u_i(k) = -K_{i,o} x(k)$ and the elements of $H_i^{(s+1)}$ are
\begin{align}\label{Hterms}
    \begin{split}
        & H_{ixx}^{(s+1)} = Q_i^{(s)} + A_i^\top P_i^{(p+1)} A_i + \sum_{j\neq i} K_{j,o}^\top R_{ij} K_{j,o}, \\ 
        & H_{ixu}^{(s+1)} = H_{iux}^{(s+1)\top} = A_i^\top P_i^{(s+1)} B_i,\\
        & H_{iuu}^{(s+1)} = R_{ii} + B_i^\top P_i^{(s+1)} B_i.
    \end{split}
\end{align}
Then, using the above, \eqref{upd1} can be rewritten as 
\begin{align}\label{upd1redone}
    \begin{split}
        & x^\top(k) (Q_i^{(s)} + \sum_{j\neq i} K_{j,o}^\top R_{ij} K_{j,o}) x(k) + u_i^\top(k) R_{ii} u_i(k) - \\ & \bar{x}_i^\top(k) H_i^{(s+1)} \bar{x}_i(k) + \bar{x}_i^\top (k+1) H_i^{(s+1)} \bar{x}_i(k+1) = 0.
    \end{split}
\end{align}
Next, we modify \eqref{upd2} in the same fashion as \eqref{upd1}. Multiplied by $x^\top(k)$ and $x(k)$, \eqref{upd2} becomes
\begin{equation}\label{upd2redone}
 x^\top(k) Q_i^{(s+1)} x(k) = x^\top(k) Q_i^{(s)}  x(k) + \alpha_i 
 x^\top(k) \Delta_i^{(s+1)}  x(k),
\end{equation}
where $\Delta_i^{(s+1)}$ is given in \eqref{ktilde}. $\Delta_i^{(s+1)}$ can be rewritten in terms of \eqref{Hterms} without using dynamics matrices as follows 
\begin{align}\label{DELTAredone}
    \begin{split}
        &\Delta_i^{(s+1)} = \\
        &((H_{iuu}^{(s+1)})^{-1} H_{iux} - K_{i,o})^\top H_{iuu}^{(s+1)} ((H_{iuu}^{(s+1)})^{-1} H_{iux} - K_{i,o}).
    \end{split}
\end{align}
Hence, solving \eqref{upd1redone} with respect to $H_i^{(s+1)}$ and then using it to update $Q_i^{(s+1)}$ enforces the result of Lemma \ref{condlem}. After the procedure is repeated a number of times that gives a desired tolerance, i.e., $\|Q_i^{(s+1)} - Q_i^{(s)}\| \leq \rho_i$ for some small constant $\rho_i > 0$, the feedback law is computed using $H_i^{(s+1)}$ computed as
\begin{equation}\label{upd3redone}
    K_i = (H_{iuu}^{(s+1)})^{-1} H_{iux}^{(s+1)}.
\end{equation}

Both \eqref{upd1redone} and \eqref{upd2redone} include trajectories $x(k)$ and $u_i(k)$. Hence these trajectories need to be generated. Solving \eqref{upd1redone} via the batch-least square, as in \cite{jiang12}, requires these trajectories to satisfy a persistence of excitation condition (PE) which can be guaranteed by injecting probing noise in the control inputs \cite{had08}. The choice of noise might be random noise \cite{al07}, exponentially decreasing exploration noise \cite{vamvs15} or sinusoidal signal with different frequencies \cite{jiang12}. Thus, the following trajectories 
\begin{equation}
    x(k+1) = A_i x(k) + B_i u_i(k),  
\end{equation}
where $u_i(k) = - K_{i,o} x(k) + \epsilon_i (k)$ with $\epsilon_i (k)$ being a noise need to be generated for each $i\in\mathcal{N}$, i.e., totally $N$ pairs of trajectories denoted as $(x_i^{\epsilon},u_i)$. Adding noise does not affect solutions as it is shown in \cite{kiumarsi17} for the linear systems.

\begin{rmk}
In contrast to \eqref{upd1redone}, solving \eqref{upd2redone} does not require injection of noise to the control inputs. However, it is not necessary to generate one more noise-free  trajectory $x(k)$ (via $x(k+1) = (A - \sum_{j=1}^N B_j K_{j,o}) x(k)$), since the trajectories generated for solving \eqref{upd1redone} can be used. 
\end{rmk} 

The proofs of the convergence and Nash optimality are omitted since they are the same as in \cite{al07}. The proofs are based on proving the equivalence of solving \eqref{upd1} and \eqref{upd2} to \eqref{upd1redone} and \eqref{upd2redone}, respectively \cite{had08}. Also, other analytical properties shown in Theorem \ref{stabth} and Proposition \ref{characprop} are also valid for the model-free algorithm.  

The way to solve \eqref{upd1redone} and \eqref{upd2redone} given the pairs of generated trajectories is described in the following section. 

\subsection{Implementation of model-free algorithm}
To implement the algorithm, one needs to use Kronecker product properties 
\begin{equation}
a^\top B c = (c^\top\otimes a^\top) \text{vec} (B). 
\end{equation} 
Then, \eqref{upd1redone} can be rewritten as 
\begin{align}
    \begin{split}
        & x_i^{\epsilon\top}(k) (Q_i^{(s)} + \sum_{j\neq i} K_{j,o}^\top R_{ij} K_{j,o}) x_i^{\epsilon}(k) + u_i^\top(k) R_{ii} u_i(k) = \\ & (\bar{x}_i^{\epsilon\top}(k)\otimes \bar{x}_i^{\epsilon\top} (k)   - \bar{x}_i^{\epsilon\top} (k+1) \otimes \bar{x}_i^{\epsilon\top}(k+1))\text{vec}(H_i^{(s+1)}),
    \end{split}
\end{align}
where $\bar{x}_i^{\epsilon}(k) = [x_i^{\epsilon}(k), u_i(k)]^\top$. $H_i^{(s+1)}$ is a symmetric matrix that has $(n+m_i)(n+m_i+1)/2$ unknown elements that are to be computed. We introduce the following notations
\begin{align}
\begin{split}
\phi_i(k) & = x_i^{\epsilon\top}(k) (Q_i^{(s)} + \sum_{j\neq i} K_{j,o}^\top R_{ij} K_{j,o}) x_i^{\epsilon}(k) + u_i^\top(k) R_{ii} u_i(k),\\
\psi_i(k) & = (\bar{x}_i^{\epsilon\top}(k)\otimes \bar{x}_i^{\epsilon\top} (k)   - \bar{x}_i^{\epsilon\top} (k+1) \otimes \bar{x}_i^{\epsilon\top}(k+1))^\top,\\
\Phi_i & = (\phi_i(k), \phi_i(k+1), \dots, \phi_i(k + n_{H_i} -1))^\top,\\
\Psi_i & = (\psi_i(k), \psi_i(k+1),\dots, \psi_i(k + n_{H_i} -1)^\top.\\
\end{split}
\end{align}
where $n_{H_i} \geq (n+m_i)(n+m_i+1)/2$. Then, one can use the batch-least square \cite{jiang12} to calculate 
\begin{equation}\label{lq1}
\text{vec}(H_i^{(s+1)}) = (\Psi_i^\top \Psi_i)^{-1} \Psi_i^\top \Phi_i.
\end{equation}

To compute update $Q_i^{(s+1)}$ in \eqref{upd2redone} one rewrites it as
\begin{align}
    \begin{split}
 &(x_i^{\epsilon\top}(k) \otimes x_i^{\epsilon\top} (k)) \text{vec} (Q_i^{(s+1)}) = \\
 & x_i^{\epsilon\top}(k) Q_i^{(s)} x_i^{\epsilon}(k) + \alpha_i x_i^{\epsilon\top}(k)\Delta_i^{(s+1)} x_i^{\epsilon}(k),
\end{split}
\end{align}
where $\Delta_i^{(s+1)}$ is given in \eqref{DELTAredone}. Then, the following notations are used 
\begin{align}
	\begin{split}
	\eta_i (k) & =  x_i^{\epsilon\top}(k) Q_i^{(s)} x_i^{\epsilon}(k) + \alpha_i x_i^{\epsilon\top}(k) \Delta_i^{(s+1)} x_i^{\epsilon}(k), \\
	\theta_i (k)  &=  (x_i^{\epsilon\top} (k) \otimes x_i^{\epsilon\top}(k))^\top,\\
	H_i & = (\eta_i(k),\eta_i(k+1),\dots,\eta_i(k+ n_Q - 1)^\top, \\
	\Theta_i & = (\theta_i(k), \theta_i(k+1), \dots, \theta_i(k+n_Q -1))^\top. 
	\end{split} 
\end{align}
where $n_Q\geq n(n+1)/2$ because $Q_i^{(s+1)}$ is symmetric and has $n(n+1)/2$ elements to estimate.  Then, one can use the batch-least square to calculate 
\begin{equation}\label{lq2}
\text{vec}(Q_i^{(s+1)}) = (\Theta_i^\top \Theta_i)^{-1} \Theta_i^\top H_i.
\end{equation}

The procedure of the model-based algorithm is summarised below in \textbf{Algorithm~ \ref{mfalg}}.

Finally, the steps of the model-free algorithm are summarized below. 
\begin{algorithm}
\caption{Model-free Algorithm}
\label{mfalg}
    \begin{enumerate}
	\item Initialize $R_{ii} > 0$, $R_{ij} \geq 0$, $Q_i^{(0)} > 0$ and the step sizes $\alpha_i\in(0,1]$ for $i,j\in\mathcal{N}$. Set the iteration counter $s = 0$, the desired tolerance $\rho_i > 0$ and collect $N$ pairs of trajectories $(x_i^{\epsilon},u_i)$ for $i\in\mathcal{N}$.\\
	\textbf{For each player $i\in\mathcal{N}$ perform the following}:
	\item Compute $H_i^{(s+1)}$ from \eqref{upd1redone} via \eqref{lq1} and update $Q_i^{(s+1)}$ as in \eqref{upd2redone} via \eqref{lq2}.  
	\item If $\|Q_i^{(s+1)} - Q_i^{(s)}\| \leq \rho_i$ then stop and compute $K_i$ from \eqref{upd3redone}. Otherwise, set $s+1$ and repeat step 2. 
	\end{enumerate}
\end{algorithm}

\section{Simulations}\label{SIMS}
In this section, we present simulation results for the introduced algorithms. 

\subsection{Model-based algorithm simulation}
For the model-based algorithm, we consider slightly modified dynamics from \cite{song21} (the plant matrix $A$ is made unstable). Consider the following discrete-time dynamics 
\begin{equation}
x(k+1) = A x(k) + \sum_{j=1}^4 B_j u_j(k),
\end{equation}
where 
\mathleft
\begin{align}
    \begin{split}
&A = \begin{pmatrix}
1.1 & 0.09983 \\ -0.09983 & 0.995
\end{pmatrix},\, B_1 = \begin{pmatrix}0.2097 \\ 0.08984\end{pmatrix},\\ 
& B_2 = \begin{pmatrix} 0.2147 \\ 0.2895 \end{pmatrix},\, B_3 = \begin{pmatrix} 0.2097 \\ 0.1897\end{pmatrix}.\, 
B_4 = \begin{pmatrix} 0.2 \\ 0.1 \end{pmatrix}.
    \end{split}
\end{align}
The observed game $(A,\mathbf{B},\mathbf{Q}_o,\mathbf{R}_o)$ cost function parameters parameters are given below 
\mathcenter
\begin{align}
    \begin{split}
    & Q_{1,o} = \begin{pmatrix} 5 & 0\\ 0 & 7\end{pmatrix},\quad 
      Q_{2,o} = \begin{pmatrix}10 & 0\\ 0 & 3\end{pmatrix}, \\
    & Q_{3,o} = \begin{pmatrix} 3 & 0\\ 0 & 1\end{pmatrix},\quad 
      Q_{4,o} = \begin{pmatrix} 1 & 0\\ 0 & 1\end{pmatrix},
    \end{split}
\end{align}
and 
\begin{align}
\begin{split}
    & R_{11,o} =  R_{22,o} = R_{33,o} = R_{44,o} = 1, \\
    & R_{12,o} = R_{14,o} = R_{23,o} = R_{31,o} = 1,
\end{split}
\end{align}
and the rest $R_{ij} = 0$ for $i,j \in\mathcal{N}$. 

The game was solved using the algorithm presented in \cite{song21}. The resulted NE tuple $(K_{1,o}, K_{2,o}, K_{3,o}, K_{4,o})$ for the described game is given below

\begin{align}
    \begin{split}
& K_{1,o} = \begin{pmatrix} 2.2058 & -0.6285\end{pmatrix},\, K_{2,o} = \begin{pmatrix} 0.3693 & 1.1207 \end{pmatrix}, \\
& K_{3,o} = \begin{pmatrix} 0.3216 & 0.1016\end{pmatrix},\, K_{4,o} = \begin{pmatrix} 0.1883 & -0.0226 \end{pmatrix}.
    \end{split}
\end{align}

We initialize the parameters in the following way 
\begin{align}
    \begin{split}
& Q_i^{(0)} = 0.1 I_2,\quad i\in\mathcal{N}\\ 
& R_{11} = 2,\, R_{22} = 0.5,\, R_{33} = 1,\, R_{44} = 4,    
    \end{split}
\end{align}
and the rest $R_{ij} = 0$ for $i,j \in\mathcal{N}$; the step sizes and the tolerance are $\alpha_i = 1$ and $\rho_i =0.001$ for $i\in\mathcal{N}$. 

The desired tolerance was reached after $531$ iterations. The simulation results are the following
\begin{align}
\begin{split}
& P_1^{(531)} = \begin{pmatrix}
41.7888 & -8.2424 \\
   -8.2424 &  2.7042
\end{pmatrix} ,\\
& P_2^{(531)}= \begin{pmatrix}
4.9813  & 0.6768 \\
    0.6768 &  2.3762
\end{pmatrix},\\
& P_3^{(531)} = \begin{pmatrix}
3.4643  & 0.6572 \\
    0.6572  & 0.1779
\end{pmatrix},\\
&P_4^{(531)} = \begin{pmatrix} 7.3840 & -0.4517 \\
   -0.4517 &  0.0901\end{pmatrix}, 
\end{split}
\end{align}
and 
\begin{align}
\begin{split}
& Q_1^{(531)} = \begin{pmatrix}
19.2252 & -1.2966 \\
   -1.2966  & 0.2215
\end{pmatrix} ,\\
& Q_2^{(531)}= \begin{pmatrix}
3.6298 &  1.1197 \\
    1.1197 &  0.5718
\end{pmatrix},\\
& Q_3^{(531)} = \begin{pmatrix}
2.8970  & 0.5067 \\
    0.5067 &  0.0996
\end{pmatrix},\\
&Q_4^{(531)} = \begin{pmatrix}5.5080 & -0.1665\\
   -0.1665  & 0.0212\end{pmatrix}, 
\end{split}
\end{align} 
where the resulted feedback laws calculated via \eqref{upd3} are
\begin{align}
    \begin{split}
& K_1 = \begin{pmatrix} 2.1898 & -0.6298\end{pmatrix} ,\, K_2 = \begin{pmatrix}
 0.3543  & 1.1193
\end{pmatrix},\\
& K_3 = \begin{pmatrix} 0.3058 &  0.1002 \end{pmatrix} ,\,K_4 = \begin{pmatrix}
0.1731 & -0.0236
\end{pmatrix}.
    \end{split}
\end{align}
The convergence of the parameters is shown on Figure \ref{fig1}. 
\begin{figure} 
    \centering
  \subfloat[\label{1a}]{%
        \includegraphics[width=1\linewidth]{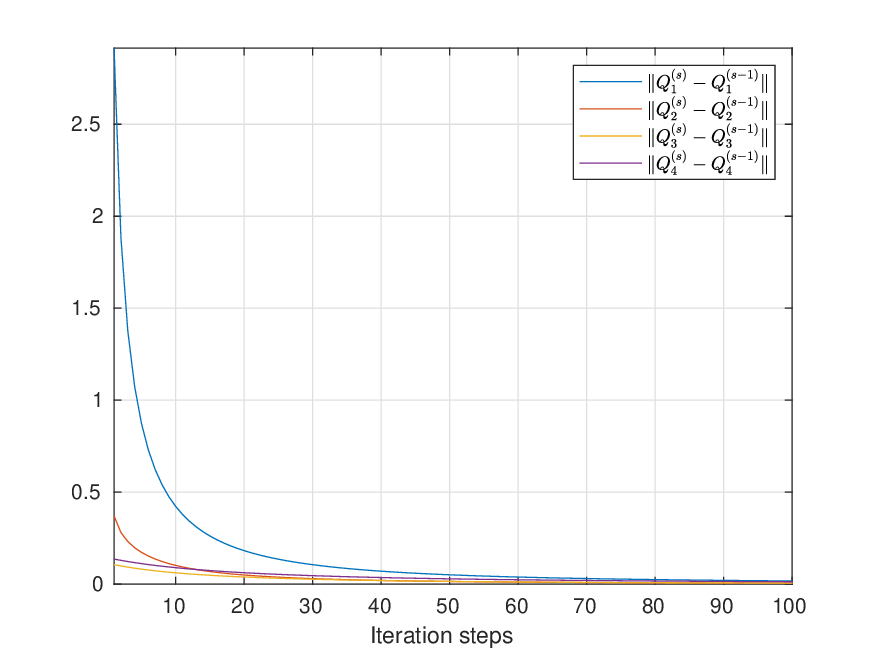}}\\
  \subfloat[\label{1b}]{%
        \includegraphics[width=1\linewidth]{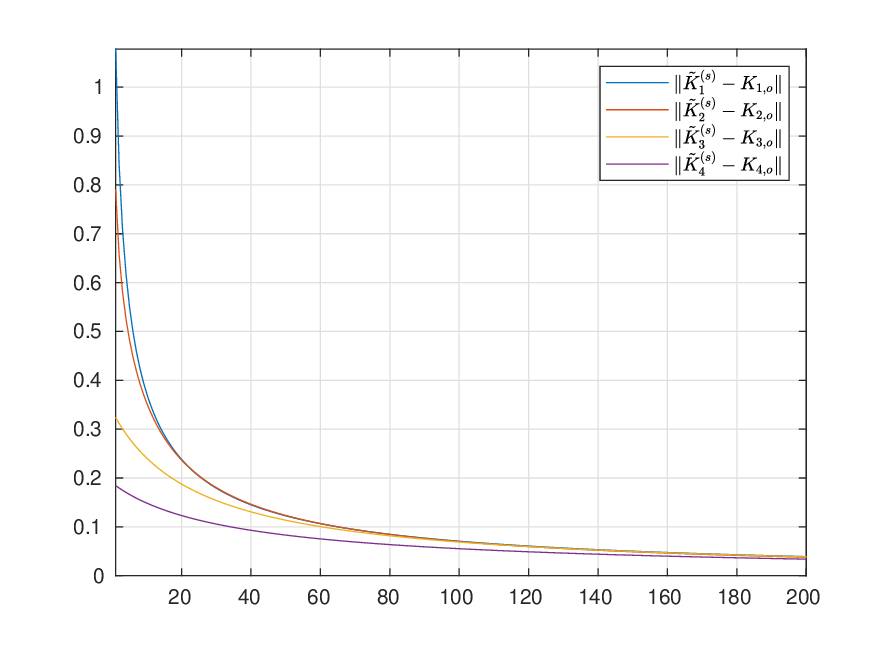}}
  \caption{Algorithm 1: (a) convergence of the norm for iterations of $Q_i^{(s)}$; (b) convergence of the norm for iterations of $\Tilde{K}_i^{(s)}$ as if it is calculated as in \eqref{ktilde}.}
  \label{fig1} 
\end{figure}
\begin{rmk}
 The reader can notice that in Figure \ref{fig1}, the iteration steps axis is shorter then the required number of iterations to reach the set tolerance, i.e., $531$. In Theorem \ref{thconv}, it is shown that the increment that one approaches the required $Q_i$ gets smaller with each iteration which results in slow convergence for a low tolerance. 

In fact, $\alpha_i $, firstly introduced in \eqref{upd21}, is designed to make every $Q_i^{(s+1)}$ to be a convex combination of $Q_i^{(s)}$ and $\Tilde{Q}_i^{(s)}$. But on the other hand, \eqref{upd2} demonstrates that with each iteration $s$, one moves closer to the required $Q_i$ with $\alpha_i \Delta_i^{(s)} >  0$ step. Thus, technically it is possible to set $\alpha_i > 1$ to improve the speed of convergence. This was checked by setting $\alpha_i = 5$ for every $i$ in the above simulation. Then, the same tolerance was reached in $260$ iterations. Obviously, the possibility of overshooting the required value for $Q_i$ then increases. 
\end{rmk}

\subsection{Model-free algorithm simulation}
Consider the following discrete-time dynamics 
\begin{equation}
x(k+1) = A x(k) + \sum_{j=1}^2 B_j u_j(k)
\end{equation}
where 
\begin{align}
    \begin{split}
       & A = \begin{pmatrix}
            0.77 & 0.36 \\ 0 & 0.85
        \end{pmatrix}, \quad B1 = \begin{pmatrix}
            0.15 \\ 0.43 
        \end{pmatrix},\quad B2 = \begin{pmatrix}
            0.17 \\ 0.31
        \end{pmatrix}.
    \end{split}
\end{align}
The observed game $(A,\mathbf{B},\mathbf{Q}_o,\mathbf{R}_o)$ cost function parameters parameters are given below 
\mathcenter
\begin{align}
    \begin{split}
    & Q_{1,o} = \begin{pmatrix} 5 & 0\\ 0 & 10\end{pmatrix},\quad 
      Q_{2,o} = 3 I_2, \\
    \end{split}
\end{align}
and 
\begin{equation}
    R_{11,o} =3,\,  R_{22,o} = 4,\, R_{12,o} = 1,\, R_{21,o} = 1.
\end{equation}

The game was solved using the algorithm presented in \cite{song21}. The resulted NE pair $(K_{1,o}, K_{2,o})$ for the described game is given below
\begin{equation}
     K_{1,o} = \begin{pmatrix}  0.1953 & 0.9638\end{pmatrix},\, K_{2,o} = \begin{pmatrix} 0.1839 & 0.2254 \end{pmatrix}.
\end{equation}

We initialize the parameters in the following way 
\begin{equation}
    Q_i^{(0)} = 0.1 I_2,\quad R_{ij} = 1,\quad i,j\in\mathcal{N}.
\end{equation}
The step sizes and the tolerance are $\alpha_i = 1$ and $\rho_i =0.001$ for $i\in\mathcal{N}$. 
\begin{rmk}
    For the model-free algorithm, trajectories must be generated with control inputs including the probing noise. For this simulation,
    $u_i (k) = -K_{i,o} x(k) + \epsilon_i(k)$ control inputs were used where $\epsilon_i(k) = 0.00005\sum_{i=1}^{10^4} \sin(\omega_i k)$ where $\omega_i$ is a scalar drawn from the standard normal distribution. Obviously, there might be a simpler choice for the probing noise. We refer the reader to \cite{ioannou2006adaptive} for more details.
\end{rmk}

The desired tolerance was reached after $183$ iterations. The simulation results are the following
\begin{align}
\begin{split}
& H_1^{(183)} = \begin{pmatrix}
0.7478 &  1.7403 & 0.3901 \\
    1.7403  &  6.6114  &  1.9877 \\
    0.3901  & 1.9877  & 2.1069
\end{pmatrix} ,\\
& H_2^{(183)}= \begin{pmatrix}
1.0220 &  1.1011 &  0.4393 \\
    1.1011 &  2.6281 &  0.5886 \\
    0.4393 &  0.5886 &  2.5689 \\
\end{pmatrix},\\
\end{split}
\end{align}
and 
\begin{align}
\begin{split}
& Q_1^{(183)} = \begin{pmatrix}
  0.4444 &  0.9810 \\
    0.9810  &   2.9255
\end{pmatrix} ,\\
& Q_2^{(183)}= \begin{pmatrix}
0.5495  &  0.5454 \\
    0.5454  & 1.0222
\end{pmatrix},\\
\end{split}
\end{align} 
where the resulted feedback laws calculated via \eqref{upd3redone} are
\begin{equation}
K_1 = \begin{pmatrix}  0.1851 & 0.9434 \end{pmatrix} ,\, K_2 = \begin{pmatrix}
  0.1710 & 0.2291
\end{pmatrix}.
\end{equation}
The convergence of the parameters is shown on Figure \ref{fig2}. 
\begin{figure} 
    \centering

       \includegraphics[width=1\linewidth]{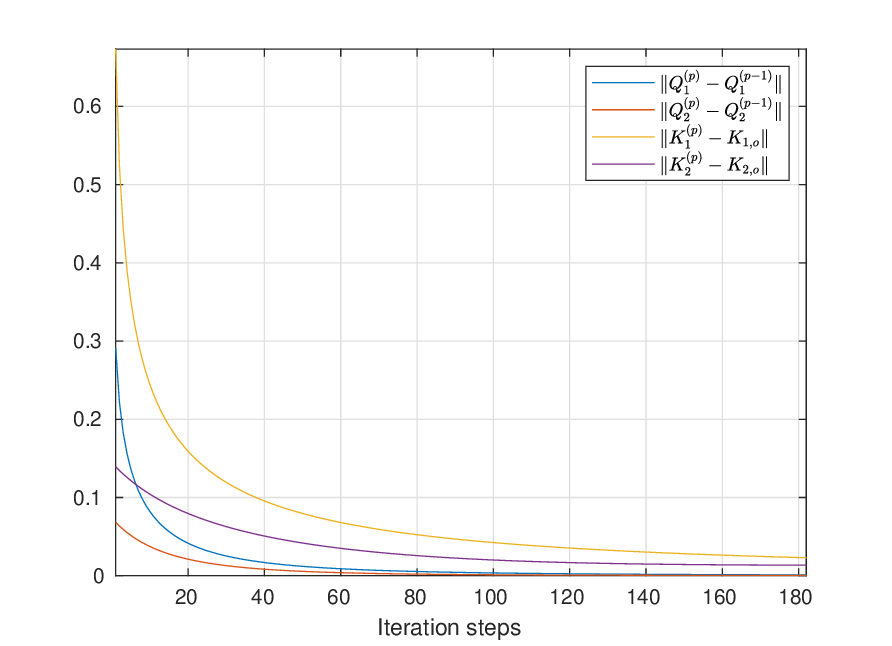}\\

  \caption{Algorithm 2: convergence of the norm for iterations of $Q_i^{(s)}$ and $\Tilde{K}_i^{(s)}$ as if it is calculated as in \eqref{upd3redone}.}
  \label{fig2} 
\end{figure}

\section{Conclusions}\label{CON}
In this paper, two algorithms to solve the inverse problem in the case of LQ discrete-time dynamic non-cooperative games are established. Firstly, we introduced the model-based algorithm and described its analytical properties. Then, the model-based algorithm is further extended to the model-free version that can solve the inverse problem in the case when the dynamics of the systems is unknown. It is shown that both the algorithms generate a set of cost function parameters that form an equivalent game. It is also shown how a new equivalent game can be generated without requiring to reuse the algorithms.

In our work we considered the closed-loop infinite horizon games with linear dynamics. Thus, finite horizon games, games with non-linear dynamics and open loop games  might be interesting to study in the context of the inverse problems. Another possible direction might be to extend the presented algorithms or introduce new ones that guarantee that the set of feedback laws of the players are stabilizing the system at each iteration. This might be useful for solving imitation problems \cite{lian2022data} where the learner imitating the expert wants to use the stabilizable set of controllers at every iteration.




\bibliographystyle{cas-model2-names}

\bibliography{cas-dc}


\bio{}
\endbio

\endbio

\end{document}